\documentclass[11pt]{amsart}

\usepackage{amscd}
\usepackage{amssymb}

\usepackage{graphicx}
\usepackage{caption}
\usepackage{subcaption}

\usepackage{times}

\usepackage{color}

\usepackage[colorlinks=true, pdfstartview=FitV, linkcolor=blue, pagebackref, citecolor=blue, urlcolor=blue]{hyperref}

\usepackage{amsthm}

\newtheorem{prop}[equation]{Proposition}

\newtheorem*{thm*}{Theorem}
\newtheorem*{prop*}{Proposition}
\newtheorem*{cor*}{Corollary}
\newtheorem*{lem*}{Lemma}
\newtheorem*{MT*}{Main Theorem}

\newtheorem*{ques*}{Question}
\newtheorem*{claim*}{Claim}

\theoremstyle{definition} %

\newtheorem*{defn*}{Definition}
\newtheorem{eg}[equation]{Example}

\theoremstyle{remark} %
\newtheorem{rmk}[equation]{Remark}

\newtheorem*{rmk*}{Remark}
\newtheorem*{rmks*}{Remarks}

\newcommand{\R}{\mathbb{R}}
\newcommand{\C}{\mathbb{C}}
\newcommand{\Z}{\mathbb{Z}}

\DeclareMathOperator{\tr}{Tr}

\newcommand{\vf}{\delta}
\newcommand{\hst}{\widetilde{\alpha}}

\newcommand{\qform}[1]{{\left\langle{#1}\right\rangle}}                   

\newcommand{\eand}{\quad\text{and}\quad}

\newcommand{\iso}{\xrightarrow{\sim}}

\DeclareMathOperator{\USp}{USp}
\DeclareMathOperator{\GL}{GL}
\DeclareMathOperator{\Spin}{Spin}
\DeclareMathOperator{\Sp}{Sp}
\DeclareMathOperator{\SU}{SU}
\DeclareMathOperator{\SO}{SO}

\newcommand{\ov}{\omega^\vee}
\newcommand{\av}{\alpha^\vee}

\newcommand{\cartan}[1]{\mathsf{#1}}
\newcommand{\At}{\cartan{A}}
\newcommand{\Bt}{\cartan{B}}
\newcommand{\Ct}{\cartan{C}}
\newcommand{\Gt}{\cartan{G}}

\newcommand{\Qv}{Q^\vee}

\newcommand{\Wa}{W_a}   

\newcommand{\slot}{\--}

\begin{document}

\title{Pictures of compact Lie groups (after Serre)}
\author{Skip Garibaldi}


\begin{abstract}
We fill in the details in a procedure outlined by Serre for drawing pictures of compact real Lie groups.   In the case of $\Sp(2n)$, the picture generated by the method is connected with abelian varieties over a number field or a finite field.  We follow the procedure to produce pictures for the three simply connected simple groups of rank 2.  The pictures for two of these have previously been discussed in the literature in a different setting.  The remaining one, type $\Gt_2$, has the most complicated picture.
 \end{abstract}

\maketitle

This note fills in the details in a procedure for drawing pictures of compact Lie groups that was outlined by J-P.~Serre in an appendix to \cite{CDSS}.  These details provide a fun illustration of fundamental results about Lie groups and root systems.

Before describing the procedure, it's worth taking a moment to reflect on the goal.  If you were going to draw a picture of a Lie group, how would you do it?  You could draw the Dynkin diagram, but that's more of an abstract representation of the group than a true picture.  Let's say the group is compact, so you can hope to capture a picture of it on a piece of paper.
Already one of the smallest Lie groups, $\Gt_2$, is 14-dimensional, so what would it mean to draw a 2-dimensional version of it?  And since locally all the points look the same\footnote{In any topological group $G$, multiplication by an element $g$ defines a homeomorphism from $G$ to $G$ that identifies the neighborhood around $1$ with the neighborhood around $g$.  So locally the neighborhood of every point looks the same.}, the picture has to be a global one.  The procedure discussed here gives a picture of a rank $n$ compact group as a compact set in $\R^n$.

To make things really concrete, we provide the pictures and details for the three rank 2 cases.  Of these three cases, two have previously appeared in the literature in other guises.  The case of $\At_2$ (i.e., $\SU(3)$) was investigated in \cite{Kaiser}.  That paper concerned average eigenvalues for elements of a compact Lie groups under its natural representation, so the resulting pictures were 1- or 2-dimensional for all compact $G$, instead of the rank of $G$.  The case of $\Ct_2$ (i.e., $\Sp(4)$) was investigated in \cite{DiPH} in the course of finding an asymptotic formula for the number of isogeny classes of abelian varieties over a finite field, and in \cite{FKRS} in connection with determining the groups that occur as the Sato-Tate group of an abelian surface over a number field.  The remaining case of type $\Gt_2$ is the most complicated.

\subsection*{Plan of the paper} 
The paper is organized as follows.  We begin with defining a continuous function $\vf \colon G \to \R^n$ for $n$ the rank of a compact Lie group $G$, in section \ref{def.sec}.  The image $\vf(G)$ is the picture of $G$ referred to in the title.  In section \ref{compute.sec}, we use invariance properties of $\vf$ to describe two ways of drawing $\vf(G)$.  One is by randomly generating points in $\vf(G)$, which we refer to as painting with a shotgun; it can easily be done with a computer.  The other requires more hand computation but allows us to write down explicit formulas.  The next three sections each work through the details of one of the three possible rank 2 semisimple and simply connected compact Lie groups:
\begin{description}
\item[\S\ref{A.sec}] Killing-Cartan type $\At_2$, i.e., $G = \SU(3)$.   See Figure \ref{a2.fig}.
\item[\S\ref{B.sec}] Killing-Cartan type $\Bt_2 = \Ct_2$, i.e., $G = \Sp(4) = \Spin(5)$.  See Figure \ref{b2.fig}.
\item[\S\ref{G.sec}] Killing-Cartan type $\Gt_2$.  This is the group of automorphisms of the octonions.  See Figure \ref{g2.fig}.
\end{description}
Finally, in section \ref{haar.sec}, we use Serre's formula for the pushforward of the Haar measure on $G$ to graphically represent that density on the picture for $\Bt_2$ and $\Gt_2$, see Figure \ref{haar.fig}.

We assume some familiarity with the notions of semisimple Lie groups, such as root systems as in \cite[Ch.~VI]{Bou:g4} and tori and semisimple groups as in \cite[Ch.~IX]{Bou:g7}.

\begin{rmk*}[added 5 November 2022]
This note was always meant to be an exposition of things known to some experts, but this turned out to be more true than expected: the papers \cite{L14}, \cite{L17}, and \cite{L19} by Gilles Lachaud contain the same pictures and more results.
\end{rmk*}

\section{The definition of \texorpdfstring{$\vf$}{d}} \label{def.sec}

We start with a compact real Lie group $G$, which we assume is semisimple and simply connected.  The aim of this section is to define a map $\vf \!: G \to \R^n$ for $n$ the rank of $G$ and to note some of its basic properties.

The complexification $G_\C$ of $G$ is a complex Lie group that is semisimple and simply connected, and it has fundamental irreducible representations $\rho_i \!: G_\C \to \GL_{d_i}(\C)$ for $i = 1, \ldots, n$.  (These are typically discussed in a textbook when it is proved that they generate the representation ring of $G_\C$.)  Recall that these are defined to be the irreducible representation whose highest weight is a fundamental dominant weight; we write $\omega_i$ for the highest weight of $\rho_i$.

Each $\rho_i$ can be one of three flavors: real, complex, or quaternionic (sometimes called ``pseudo-real'')  \cite[Chap.~IX, App.~II]{Bou:g7}.  There is a criterion in terms of weights to distinguish these cases, namely that $\rho_i$ is complex if and only if $-w_0 \omega_i \ne \omega_i$, where $w_0$ is the longest element of the Weyl group \cite[\S{IX.7.2}, Prop.~1b]{Bou:g7}.

If $\rho_i$ is not complex,  then the composition 
\[
\chi_{\rho_i} \colon G \to G_\C \xrightarrow{\rho_i} \GL_{d_i}(\C) \xrightarrow{\tr} \C
\]
has image in $\R$, and not just $\C$.  (This is not surprising if $\rho_i$ is real, because in that case $\rho_i$ is obtained by complexifying  a representation $G \to \GL_{d_i}(\R)$.)  In this way, each $\rho_i$ that is not complex gives a continuous map $\chi_{\rho_i} \colon G \to \R$.
If $G$ is $\Sp(4)$ or $\Gt_2$, then neither of the fundamental irreducible representations are complex, and we define $\vf \colon G \to \R^2$ as $\chi_{\rho_1} \oplus \chi_{\rho_2}$.

If $\rho_i$ is complex, then composing $\rho_i$ with complex conjugation $G_\C \to G_\C$ gives a different fundamental irreducible representation $\rho_j$.  In this case,we arbitrarily pick $\chi_{\rho_i}$ and use the map $G \xrightarrow{\chi_{\rho_i}} \C \iso \R^2$, i.e., $g \mapsto (\Re \chi_{\rho_i}(g), \Im \chi_{\rho_i}(g))$.  If we had chosen instead $\chi_{\rho_j}$, the second coordinate would only differ by a sign, and we ignore this difference.  The two fundamental irreducible representations of $\SU(3)$ are interchanged by complex conjugation, so this process defines a map $\vf \!: \SU(3) \to \R^2$.

In the general case, where $G$ need not have rank 2, we define $\vf$ to be the direct sum of the maps $\chi_{\rho_i}$ defined in the two previous paragraphs.

\begin{eg} 
If $\rho_i$ is not complex, then $\chi_{\rho_i}(1_G) = d_i$.  If $\rho_i$ is complex, then $\chi_{\rho_i}(1_G) = (d_i, 0)$.  Thus we have computed $\vf(1_G)$.
\end{eg}

\begin{eg} \label{rootsum}
Because $G$ is compact, every representation $\rho$ is unitary, so for every $g \in G$, $\tr \rho(g)$ is a sum of $\dim \rho$ roots of unity.  Such a sum has maximum real value $\dim \rho$, obtained when all the roots of unity equal 1.  That is, $\max_{g \in G} \Re \tr \rho(g) = \dim \rho$, and that maximum is attained only for $g \in \ker \rho$.
\end{eg}

\subsection*{Visualizing the center}
Suppose $z$ is in the center $Z(G)$ of $G$.  
Since each representation $\rho_i$ is irreducible, the element $\rho_i(z)$ is a scalar matrix, so for $g \in G$ we have:
\[
\chi_i(zg) = \tr \rho_i(zg) = \rho_i(z) \tr \rho_i(g) = \rho_i(z) \chi_i(g).
\]
In this way,
$Z(G)$ acts as symmetries on $\vf(G)$.  
Moreover, this action is faithful.  Indeed, if $z \in Z(G)$ acts trivially on $\vf(G)$, then $\rho_i(z) = 1$ for all $i$, so all weights vanish on $z$ and $z = 1$.

If we have explicit realizations of the $\rho_i$ and of $Z(G)$, we may already know the image of a given $z\in Z(G)$ under each $\rho_i$.  Whether or not we have such explicit realizations, 
we can write down elements of $Z(G)$ explicitly in terms of the coroots (viewed as homomorphisms $\C^\times \to T_\C$ for a maximal torus $T$ in $G$) as in \cite[p.~298, Table 3]{OV1}, from which we obtain the image under $\rho_i$ and can calculate the action of $Z(G)$ on $\vf(G)$.

\begin{eg}[Rank 1] \label{rank1}
In the smallest case, where $G$ has rank 1,  $G = \SU(2)$ has type $\At_1$.  It has a single fundamental irreducible representation, the tautological inclusion $\rho \colon \SU(2) \to \GL_2(\C)$.  This representation is quaternionic, and $\vf := \tr \rho$ has image in $\R$.  Because $G$ is compact and connected, the image will be a closed interval.  Which one?

By Example \ref{rootsum}, the maximum value of $\vf(G)$ is $\tr \rho(1_G) = \dim \rho = 2$.  The non-identity element of the center, $-I_2$, acts by multiplication by $-1$ on $\vf(G)$, so $-2$ is the minimum value of $\vf(G)$.  Since $\vf(G)$ is a closed interval, we conclude that $\vf(G) = [-2, 2]$.
\end{eg}

\section{Drawing \texorpdfstring{$\vf(G)$}{d(G)}} \label{compute.sec}

Each map $\chi_\rho$ as in the definition of $\vf$ has the property that 
\[
\chi_\rho(hgh^{-1}) = \tr \rho(h) \rho(g) \rho(h)^{-1} = \chi_\rho(g) \quad \text{for $g, h \in G$}.
\]
That is, $\chi_\rho$ factors through the quotient space $G/G$ of conjugacy classes in $G$.  Intuition from undergraduate linear algebra suggests this is an attractive property: it feels ``basis independent''.  And in some sense, that is all that's happening, as the following proposition shows.

\begin{prop} \label{GG.prop}
$\vf$ is a homeomorphism of $G/G$ onto a compact subset of $\R^n$.
\end{prop}

\begin{proof}
We check injectivity.  If $g_1, g_2 \in G$ satisfy $\vf(g_1) = \vf(g_2)$, then the characters of all the fundamental irreducible representations agree on $g_1$ and $g_2$, so all characters agree on $g_1$ and $g_2$.  It follows from the density of the characters in the space of square-integrable class functions that $g_1$ and $g_2$ are conjugate.

Since $\vf$ is a bijection $G/G \xrightarrow{\sim} \vf(G)$ and $G/G$ is compact, $\vf$ is a homeomorphism as claimed.
\end{proof}

\subsection*{Painting with a shotgun}
Let $T$ be a maximal torus in $G$.  Every element of $G$ is conjugate to an element of $T$ \cite[\S{IX.2.2}, Th.~2]{Bou:g7}, so $\vf(T) = \vf(G)$.  Therefore, if we want to draw $\vf(G)$, it suffices to produce elements $t\in T$ and compute $\vf(t)$.

The collection of algebraic homomorphisms from $\C^\times$ to the complexification $T_\C$ of $T$ is an abelian group which, because $G$ is simply connected, is identified with the coroot lattice $\Qv$.   (If we write $R$ for the set of roots of $G$ with respect to $T$, i.e., the nonzero weights of the action of $T$ on the Lie algebra of $G$, then Bourbaki writes $Q(R^\vee)$ in \cite{Bou:g4} for what we are calling $\Qv$.)  Setting $V = \R \otimes \Qv$, we may view $V$ as the Lie algebra of $T$.  The exponential map 
\[
\exp \!: V \to T
\]
is a surjection of groups such that
\[
\omega(\exp(v)) = e^{2\pi i \qform{\omega, v}} \quad \text{for $v \in V$ and $\omega$ a weight}.
\]

Now, since $\qform{\omega, \Qv} \subseteq \Z$ for all $\alpha \in R$, $\exp$ factors through $V/\Qv$.  Since we have picked a basis $\alpha_1, \ldots, \alpha_n$ of the root lattice, this produces a corresponding basis $\alpha_1^\vee, \ldots, \alpha_n^\vee$ of $\Qv$, which itself identifies $V/\Qv$ with $[0,1)^n$, providing an isomorphism
\begin{equation} \label{uni}
[0,1)^n \iso T \quad \text{via} \quad (x_1, \ldots, x_n) \mapsto \exp(\sum x_i \alpha_i^\vee)
\end{equation}
which we also denote by $\exp$.
We can use the Weyl Character Formula to write down the character of each of the representations used to define $\vf$, which allows us to evaluate $\vf(\exp x)$ for any particular $x \in [0,1)^n$.  Repeatedly picking $x$ uniformly at random and drawing the point $\vf(\exp x)$ draws an approximation of $\vf(G)$.  Because this amounts to randomly filling in dots in $\vf(G)$, it might be described as ``painting $\vf(G)$ with a shotgun''.  We did this for the three examples of rank 2 groups we are considering, and the resulting elements of $\vf(G)$ are depicted as blue dots in Figures \ref{a2.fig}--\ref{g2.fig}.


\subsection*{Action by the affine Weyl group}
Above, we used the fact that every element of $G$ is conjugate to an element of $T$ to see that $\vf(T) = \vf(G)$.  Yet we can say something more precise.  The inclusion $T \subset G$  induces a map $T/N_G(T) \to G/G$, and it is famously true that this map is also a homemorphism \cite[\S{IX.2.2}, Cor.~8]{Bou:g7}.  Combined with Proposition \ref{GG.prop}, this identifies $\vf(G)$ with $T/N_G(T)$.  We now explain how to use this fact to draw the boundaries of $\vf(G)$.

The map $\exp$ is invariant under the Weyl group $W = N_G(T)/T$ of $G$, so we may restrict $\exp$ to a fundamental domain for the action of $W$ on $V$ (a ``chamber''), whose closure is
\[
C := \{ v \in V \mid \text{$\qform{\alpha, v} \ge 0$ for all $\alpha \in R$} \}.
\]
We have also observed that $\exp$ is invariant under translating $V$ by elements of $\Qv$, i.e., $\exp$ is invariant under the action of the \emph{affine Weyl group} 
\[
\Wa := W \rtimes \Qv.
\]
So it further makes sense to restrict our attention to a fundamental domain for the action of $\Wa$ on $V$ (an ``alcove''), whose closure is
\[
A = C \cap \{ v \in V \mid \qform{\hst, v} \le 1 \},
\]
see for example \cite[\S{VI.2.3}]{Bou:g4}.  A diagram of $A$ in the case of type $\Gt_2$ is exhibited below in Figure \ref{G2.alcove}.
Pictures of alcoves for the three rank 2 root systems can be found in 
in Plate X of \cite{Bou:g4} or Figures 1--4 in \cite[p.~224]{Ca:simple}\footnote{This 1927 article by Cartan is the one where he classified the elements of finite order in $G$ by a method nowadays called ``Kac coordinates''.  See \cite{Reeder:tor} for full details.}.

In summary, we have continuous maps
\[
A \hookrightarrow [0,1)^n \xrightarrow{\sim} T \to T/N_G(T).
\]
The composition is a bijection, therefore a homeomorphism, compare \cite[\S IX.5.2, Cor.~1]{Bou:g7}.  

\begin{eg}[Rank 1 revisited] \label{rank1.2}
Let's return to the case $G = \SU(2)$ from Example \ref{rank1} and describe $\vf(G)$ by working through the above construction.  The root system has one simple root $\alpha = \hst$, a corresponding simple coroot $\av$, and a fundamental weight $\omega$ such that $\qform{\omega, \av} = 1$.  It follows that $A = [0, 1/2]$.  The representation $\rho$ has weights $\pm \omega$, each with multiplicity 1, so for $a \in A$ we have
\[
\vf(\exp(a)) = e^{2\pi i a} + e^{-2\pi i a} = 2 \cos(2\pi a).
\]
Again we find that $\vf(G) = [-2, 2]$.

For yet a third view on this example, see Exercise 35 on page 49 of \cite{Rees}.
\end{eg}

\subsection*{The alcove \texorpdfstring{$A$}{A} in the rank 2 case}
The rest of this paper focuses on groups of rank 2, so consider the case where $R$ has two simple roots $\alpha_1, \alpha_2$.  In this case, $A$ is a triangle with three boundary lines.
One vertex of $A$, the intersection of the lines $\qform{\alpha_1, \slot} = 0$ and $\qform{\alpha_2, \slot} = 0$, is 0.  

Another vertex is the intersection of $\qform{\alpha_1, \slot} = 0$ with $\qform{\hst, \slot} = 1$.  This is $\ov_2/a_2$, where $\ov_2$ is the fundamental weight of the dual root system $R^\vee$ such that $\qform{\ov_2, \alpha_1} = 0$ and $\qform{\ov_2, \alpha_2} = 1$ and $a_i$ is the coefficient of $\alpha_i$ in the expression $\hst = a_1 \alpha_1 + a_2 \alpha_2$ \cite[\S{VI.2.2}, Cor.]{Bou:g4}.  The third vertex of the triangle is the intersection of $\qform{\alpha_2, \slot} = 0$ and $\qform{\hst, \slot} = 1$, so $\ov_1/a_1$.  That is, the two sides of $A$ that include 0 are $\{ t \ov_i \mid 0 \le t \le 1/a_i \}$ for $i = 1, 2$.

\section{The picture for \texorpdfstring{$\SU(3)$}{SU(3)} (type \texorpdfstring{$\At_2$}{A2})} \label{A.sec}

The compact simply connected group of type $\At_2$, $\SU(3)$, consists of 3-by-3 complex matrices $X$ such that $XX^* = I_3$ and $\det X = 1$.  One of the fundamental representations is  the (tautological) map $\rho\!: \SU(3) \to \GL_3(\C)$, and we have defined $\vf \!: G \to \R^2$ to be $g \mapsto (\Re \tr \rho(g), \Im \tr \rho(g))$.  

We have $\vf(1_G) = (3, 0)$.  Also, the center of $\SU(3)$ is generated by $e^{2\pi i/3} I_3$, so the image $\vf(G)$ is invariant under rotations by 120$^\circ$.

The boundary of the alcove joining $0$ to each of the other two vertices is $\{ t \omega_i^\vee \mid 0 \le t \le 1 \}$ for $i = 1, 2$, where 
\[
\omega_1^\vee = \frac13 (2 \alpha^\vee_1 + \alpha^\vee_2) \eand
\omega_2^\vee = \frac13 (\alpha^\vee_1 + 2\alpha^\vee_1).
\]

We next translate this into elements of the torus $T$.
Concretely, $T$ can be taken to be the collection of diagonal $X$, where $h_1(z) \in T$ has diagonal entries $z, 1/z, 1$ and $h_2(z) \in T$ has diagonal entries $1, 1/z, z$ for $z \in S^1$.  The boundary curve, then, consists of the elements $h_1(z) h_2(z)^2$, for which
\[
\tr \rho (h_1(z) h_2(z)^2) =  2z + z^{-2} \quad \text{for $z \in S^1$}.
\]
The boundary curve is therefore given by 
\begin{equation} \label{a.bd1}
x = 2 \cos(t) + \cos(2t) \eand y = 2\sin(t) - \sin(2t) \quad \text{for $0 \le t \le 2\pi/3$.}
\end{equation}
This is the upper border of Figure \ref{a2.fig} joining $(3,0)$ to $(-3/2, 3\sqrt{3}/2)$.  

\begin{figure}[hbtp]
\includegraphics[width=2.5in]{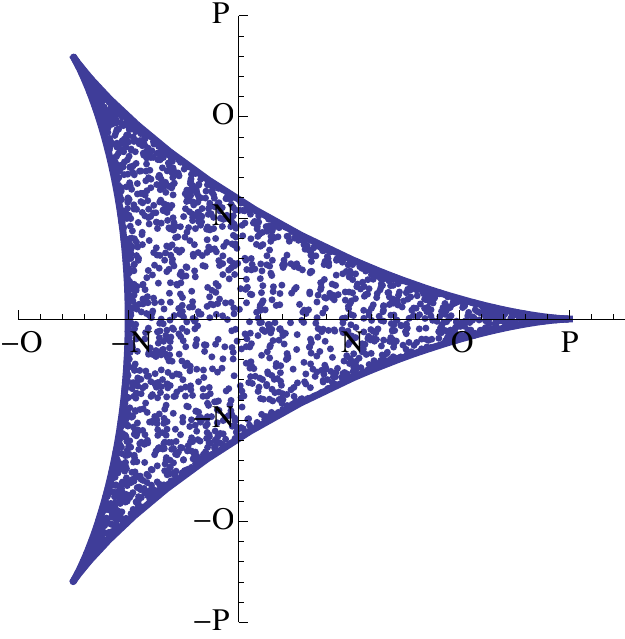}
\caption{Group $\SU(3)$ of type $\At_2$, a hypocycloid.  The vertices are at $(3, 0)$ and $(-3/2, \pm 3\sqrt{3}/2)$.} \label{a2.fig}
\end{figure}

The other two arcs making up the boundary of the figure are necessarily the images of this one under iterated rotation by 120$^\circ$; in fact these are given by Equation \eqref{a.bd1} where $t$ ranges from 0 to $2\pi$.  The resulting figure is a deltoid, also known as a \emph{hypocycloid of three cusps}, the curve traced by a point on the boundary of a circle of radius 1 rolling around the inside of a fixed circle of radius 3.

\begin{rmk*}
In the figure, the symmetry across the horizontal axis reflects the compatibility of $\vf$ with complex conjugation, i.e., $\vf(\overline{g})$ is the reflection about the horizontal axis of $\vf(g)$.
\end{rmk*}

\begin{rmk*}
For any $n$, \cite{Kaiser} draws a 2-dimensional picture of $\SU(n)$ by applying the map $g \mapsto (\Re \tr g, \Im \tr g)$ used here for $\SU(3)$.  As in the preceding remark, since $\SU(n)$ has center the $n$-th roots of unity, the resulting figure will be symmetric under  rotations by $360^\circ/n$.  Indeed, one finds a hypocycloid with $n$ cusps.  Interestingly, embedding $\SU(n)$ in $\SU(n+1)$ allows one to see that each hypocycloid should nest in the next, resulting in some impressive animations such as at \cite{Baez:rolling}.
\end{rmk*}


\section{The picture for \texorpdfstring{$\Bt_2 = \Ct_2$}{B2 = C2}} \label{B.sec}
The compact simply connected form of type $\Bt_2 = \Ct_2$ is usually denoted $\Spin(5)$ or $\Sp(4)$ or $\USp(4)$.  
Here the two fundamental irreducible representations are $\rho_1 \!: G \to \GL_4(\C)$ (the natural representation of $\Sp(4)$ or the spin representation of $\Spin(5)$, which is quaternionic) and $\rho_2 \!: G \to \GL_5(\R)$ (the natural representation of $\SO(5)$).  We define $\vf$ via $\vf(g) = (\tr \rho_1, \tr \rho_2)$.  

The center of $\Sp(4)$ is generated by $-1$, which acts trivially on the 5-dimensional representation --- i.e., $\rho_2(-1) = 1$ --- so it reflects $\vf(G)$ about the vertical axis.  Therefore, while the origin of the alcove maps to $\vf(\exp 0) = \vf(1_G) = (4, 5)$, one of the other vertices of the alcove maps to $(-4, 5)$ and the third maps to a point with $x$-coordinate zero.  By the way, this symmetry around the vertical axis and Example \ref{rootsum} combine to give that the collection $\chi_1(G)$ of $x$-coordinates of points of $\vf(G)$ is $[-4, 4]$.

\begin{figure}[hbtp]
\includegraphics[width=2.5in]{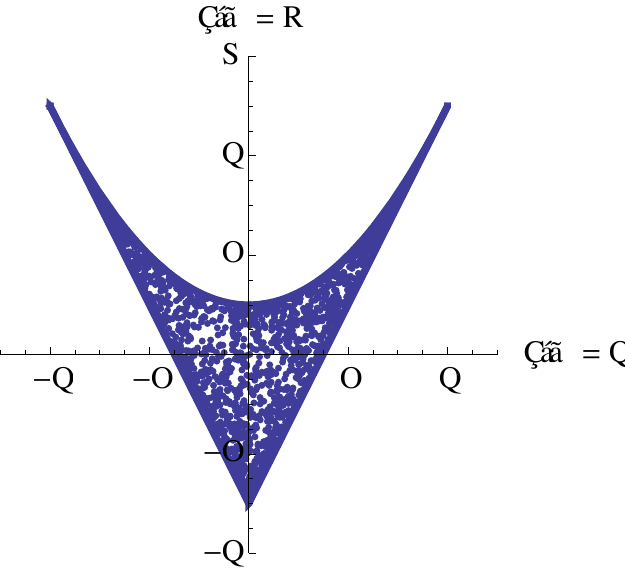}
\captionof{figure}{Group $\Sp(4) = \Spin(5)$ of type $\Ct_2 = \Bt_2$, a ``goatee''.  The vertices are at $(\pm 4, 5)$ and $(0, -3)$.} \label{b2.fig}
\end{figure}


\setlength{\unitlength}{0.05cm}
Let us find the walls of the alcove and thereby calculate the curves bounding $\vf(G)$.
We number the simple roots of $G$ according to the labeling
\[
\begin{picture}(7,2)(0,0)
\put(2,1){\circle*{3}}
\put(2,0.5){\line(1,0){15}}
\put(2,1.5){\line(1,0){15}}
\put(6,-0.5){{\small\mbox{$<$}}}
\put(17,1){\circle*{3}}

\put(-3,-0.5){\mbox{\tiny $1$}}
\put(20,-0.5){\mbox{\tiny $2$}}
\end{picture}
\]
which is the numbering for type $\Ct_2$ in \cite{Bou:g4}.
Using the Weyl character formula, we find that
\begin{align*}
\text{$\rho_1$ has weights} & \pm \omega_1,\ \pm  (\omega_2 - \omega_1) \\
\text{$\rho_2$ has weights} & \pm \omega_2,\ \pm(2\omega_1 - \omega_2), \ 0.
\end{align*}

The inverse root system has
\[
\ov_1 = \av_1+ \av_2 \eand
\ov_2 = \textstyle\frac12 \av_1 + \av_2.
\]
The highest root of $\Ct_2$ is $\hst = 2\alpha_1 + \alpha_2$, so the walls of the alcove through 0 are
\begin{equation} \label{c.wall}
\{ t \ov_1 \mid 0 \le t \le \textstyle\frac12 \} \eand \{ t \ov_2 \mid 0 \le t \le 1 \}.
\end{equation}

Now $\qform{\omega_1, \ov_1} = \qform{\omega_2, \ov_1} = 1$ and employing the identity $e^{i\theta} + e^{-i\theta} = 2 \cos \theta$, we find that
\[
\chi_1(\exp t \ov_1) = 2 \cos(2 \pi t) + 2 \eand
\chi_2(\exp t \ov_1) = 4 \cos(2 \pi t) + 1.
\]
Consequently, the image of the the first wall in \eqref{c.wall} is a line segment joining $(0, -3)$ to $(4, 5)$.

For the second wall in \eqref{c.wall}, we have $\qform{\omega_1, \omega_2^\vee} = 1/2$ and $\qform{\omega_2, \omega_2^\vee} = 1$, so
\[
x(t) = \chi_1(\exp t \ov_2) = 4 \cos(\pi t) \eand
y(t) = \chi_2(\exp t \ov_2) = 3 + 2 \cos(2 \pi t)
\]
for $0 \le t \le 1$.
Eliminating $\cos(\pi t)$ from these equations, we find the equation of a parabola $y = x^2/4 + 1$ for $-4 \le x \le 4$.

By symmetry across the vertical axis, it follows that the third wall,
the image of the line segment joining $\omega_1^\vee/2$ to $\omega_2^\vee$, maps to a line segment joining $(0, -3)$ to $(-4, 5)$.

\begin{rmk*} \label{FKRS.rmk}
Since $\rho_1$ is injective, we may identify $\Sp(4)$ with its image under $\rho_1$, which (up to a choice of basis) is the group of 4-by-4 unitary matrices that are also skew-symmetric.  Trivially, for $g \in \Sp(4)$, $\chi_1(g)$ is the linear coefficient of the characteristic polynomial of $g$.  Examining the weights of $\rho_2$ shows that $\chi_2(g)+1$ is the quadratic coefficient.  In this way, we find that the region in Figure \ref{b2.fig} was already described in \cite[Example 2.1.2]{DiPH} and \cite[p.~1425]{FKRS}.
\end{rmk*}

\section{The picture for \texorpdfstring{$\Gt_2$}{G2}} \label{G.sec}
We now turn to the case where $G$ is $\Gt_2$.  Number the simple roots of $G$ as in the \cite{Bou:g4}, i.e., 
\[
\begin{picture}(7,2)(0,0)
\put(2,1){\circle*{3}}
\put(2,0.1){\line(1,0){15}}
\put(2,1.1){\line(1,0){15}}
\put(2,2.1){\line(1,0){15}}
\put(6,-0.5){{\small\mbox{$<$}}}
\put(17,1){\circle*{3}}

\put(-3,-0.5){\mbox{\tiny $1$}}
\put(20,-0.5){\mbox{\tiny $2$}}
\end{picture}
\]
With this numbering, the action of $G$ on the trace zero octonions is the fundamental irreducible representation $\rho_1 \colon G \to \GL_7(\R)$ and the adjoint representation is the fundamental irreducible representation $\rho_2 \colon G \to \GL_{14}(\R)$.

In this notation, \cite[5.5]{Katz:G2} and \cite[eq.~(37)]{Kaiser} previously noted that $\chi_1(G) =  [-2,7]$.  From another perspective, a result in \cite{Se:trace} says that $\min_{g \in G} \chi_2(g) = -2$.  Both of these statements are visible in Figure \ref{g2.fig}.
\begin{figure}[hbtp]
\includegraphics[width=2.5in]{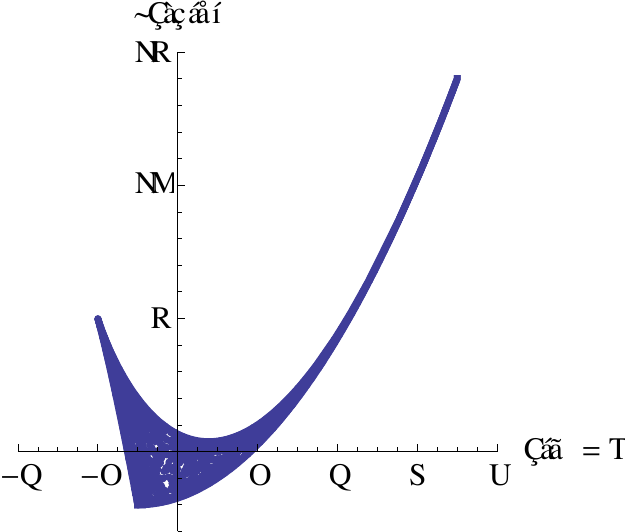} 
\caption{Group $\Gt_2$, a ``tomoe''.  The vertices are at $(7, 14)$, $(-2, 5)$, and $(-1, -2)$.} \label{g2.fig}
\end{figure}

We find that
\begin{align*}
\text{$\rho_1$ has weights} & \pm \omega_1,\ \pm  (-\omega_1 + \omega_2), \ \pm (2\omega_1 - \omega_2),\ 0 \\
\text{$\rho_2$ has weights} & \begin{cases}
\pm \omega_2,\ \pm(3\omega_1 - \omega_2), \ \pm\omega_1,\ \pm(-\omega_1 + \omega_2),\\
\pm(2 \omega_1 - \omega_2), \ \pm (-3 \omega_1 + 2\omega_2), \ 0,\  0.
\end{cases}
\end{align*}
The root system $R$ is isomorphic to its inverse root system $R^\vee$.  We have:
\[
\ov_1 = 2\av_1 + 3 \av_2 \eand
\ov_2 = \av_1 + 2 \av_2.
\]
The highest root is $\hst = 3\alpha_1 + 2\alpha_2$, so the walls of the alcove through 0 are
\begin{equation} \label{g.wall}
\{ t \ov_1 \mid 0 \le t \le \textstyle\frac13 \} \eand \{ t \ov_2 \mid 0 \le t \le \textstyle\frac12 \}.
\end{equation}
In Figure \ref{G2.alcove}, we exhibit the alcove $A$ as a blue triangle.  The other two triangles are also fundamental domains for the affine Weyl group $W_a$ acting on $V$, and reflecting $A$ first to the green triangle and then to the brown triangle, we see that $W_a$ identifies the third wall of the alcove (joining $\ov_1/3$ and $\ov_2/2$) with the line segment joining $\ov_1/3$ to the point labeled $\ov_1/2$ in the figure.  (Note that that point is labeled correctly, because $\ov_1 = \ov_2 + (\alpha_1^\vee + \alpha_2^\vee)$, i.e., the points 0, $\ov_1$, $\ov_2$, $\alpha_1^\vee + \alpha_2^\vee$ are the corners of a parallelogram and the point labeled $\ov_1/2$ is indeed the midpoint of a diagonal.)
This observation about the third wall tells us that replacing the range $0 \le t \le \tfrac13$ in the first set in \eqref{g.wall} with $0 \le t \le \tfrac12$ and applying $\vf \circ \exp$ will produce all three boundary curves for $\vf(G)$.  (We could have employed a similar maneuver also for the other two groups.)

\begin{figure}[hbt]
\includegraphics{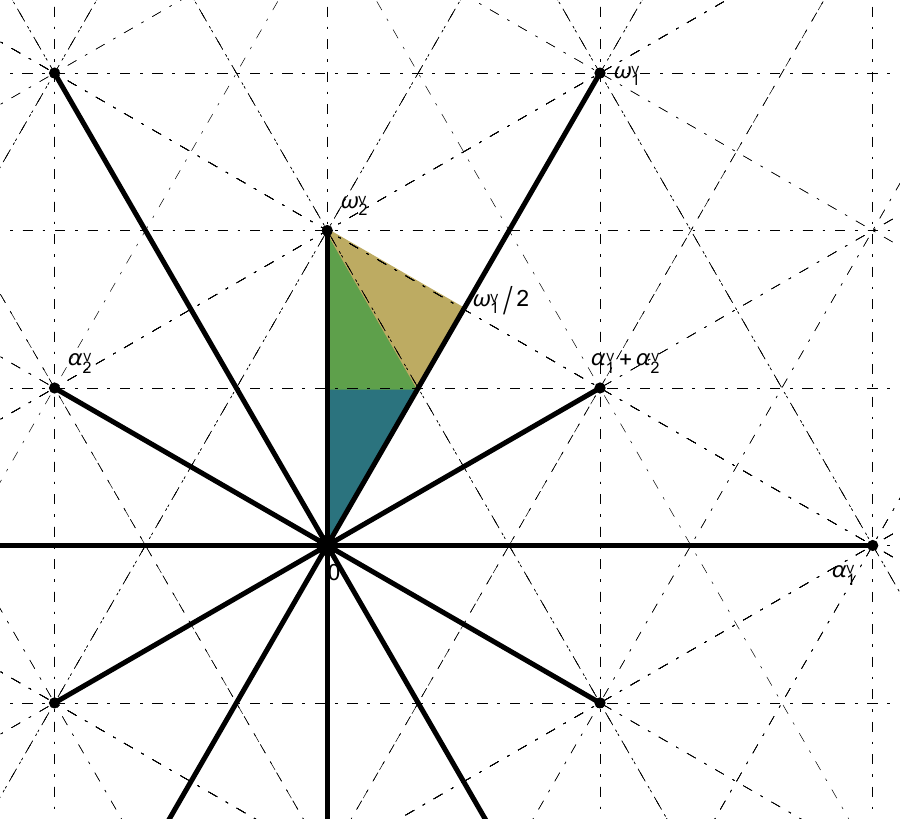}
\caption{A diagram of $V$, the coroot space, for $\Gt_2$, with emphasis on the details around 0 and $\omega_2^\vee$.  Black dots indicate the roots and a circle marks the origin.  A solid line joins the origin to each root.  Dashed lines represent lines of symmetry with respect to the affine Weyl group $W_a$, and the alcove $A$ is the blue triangle.  The other colored triangles are examples of images of $A$ under the group action.} \label{G2.alcove}
\end{figure}

Starting then with that extended wall, we find 
\begin{align*}
\chi_1(\exp t\ov_1) &= 1 + 4 \cos u + 2 \cos(2 u) = 4 (\cos u)^2 + 4 \cos u - 1 \quad \text{and}  \\
\chi_2(\exp t\ov_1) &= 4 + 4 \cos u + 2 \cos(2u) + 4 \cos(3u) \\ &= 16 (\cos u)^3 + 4 (\cos u)^2 - 8 \cos u + 2
\end{align*}
for $u = 2\pi t$, so $0 \le u \le \pi$.  This curve passes from the upper right at $(7, 14)$ at $u = 0$, to the upper left at $(-2, 5)$ at $u = 2 \pi / 3$, and down to $(-1,-2)$ at $u = \pi$.

\begin{rmk}
Setting $x(t) = \chi_1(\exp t \ov_1)$ and $y(t) = \chi_2(\exp t\ov_1)$ and eliminating $\cos(u)$ from the 
equations gives
\[
y^2 + 10y - 7 = 4x^3 -x^2 - 2x - 10xy.
\]
This resembles the equation for an elliptic curve, but it's not because this one is singular --- an ordinary cusp --- at $(x,y) = (-2,5)$.
Solving the equation for $y$ gives
\[
y = -5 - 5x \pm 2 \sqrt{8 + 12x + 6x^2 + x^3}.
\]
The ``line" joining $(-2, 5)$ with $(-1, -2)$ has equation
\[
y = -5 - 5x - 2 \sqrt{8 + 12x + 6x^2 + x^3} \quad \text{for $-1 \le x \le -2$.}
\]
And the curve joining $(-2, 5)$ with $(7, 14)$ has equation
\[
y = -5 - 5x + 2 \sqrt{8 + 12x + 6x^2 + x^3} \quad \text{for $-2 \le x \le 7$.}
\]
\end{rmk}

What about the final boundary curve, the image of the second wall in \eqref{g.wall}?  We have:
\begin{align*}
\chi_1(\exp t\ov_1) &= 3 + 4 \cos u \quad \text{and}  \\
\chi_2(\exp t\ov_1) &= 4 + 8 \cos u + 2 \cos(2u) = 4 (\cos u)^2 + 8 \cos u + 2
\end{align*}
for $u = 2\pi t$, so $0 \le u \le \pi$.  Eliminating $\cos u$ we find
\[
y = \frac14 (x^2 + 2x - 7) \quad \text{for $-1 \le x \le 7$,}
\]
a parabola.

Overall, the shape in Figure \ref{g2.fig} resembles a tomoe, a comma-shaped figure that appears more commonly doubled (futatsudomoe) or tripled (mitsudomoe).

\section{Haar measure} \label{haar.sec}

The groups $G$ considered here have an invariant measure, the \emph{Haar measure}.  By combining calculations of Steinberg and Weyl, Serre provided in \cite[Th.~A.2]{CDSS} an explicit formula for the density $\varphi$ (with respect to the standard measure) of the image of the Haar measure under $\vf$, under the assumption that $-1$ is in the Weyl group.  Specifically, for $t \in T$, we have
\begin{equation} \label{varphi.def}
\varphi(\vf(t)) = \frac{\sqrt{\left|\prod_{\alpha > 0} (\alpha(t) + 1/\alpha(t) - 2)\right|}}{(2\pi)^n}
\end{equation}
where the product is over positive roots $\alpha$.

With this formula in hand, we can calculate $\varphi$ explicitly at any given point using a computer.  Figure \ref{haar.fig} shows a contour plot for the values of $\varphi$ for $G = \Sp(4)$ and $\Gt_2$.  (A 3-dimensional plot of the $\Sp(4)$ picture appears as the right panel in \cite[p.~1432, Figure 2]{FKRS}, cf.~Remark \ref{FKRS.rmk}.)
\begin{figure}[hbt]
\includegraphics[height=2.2in]{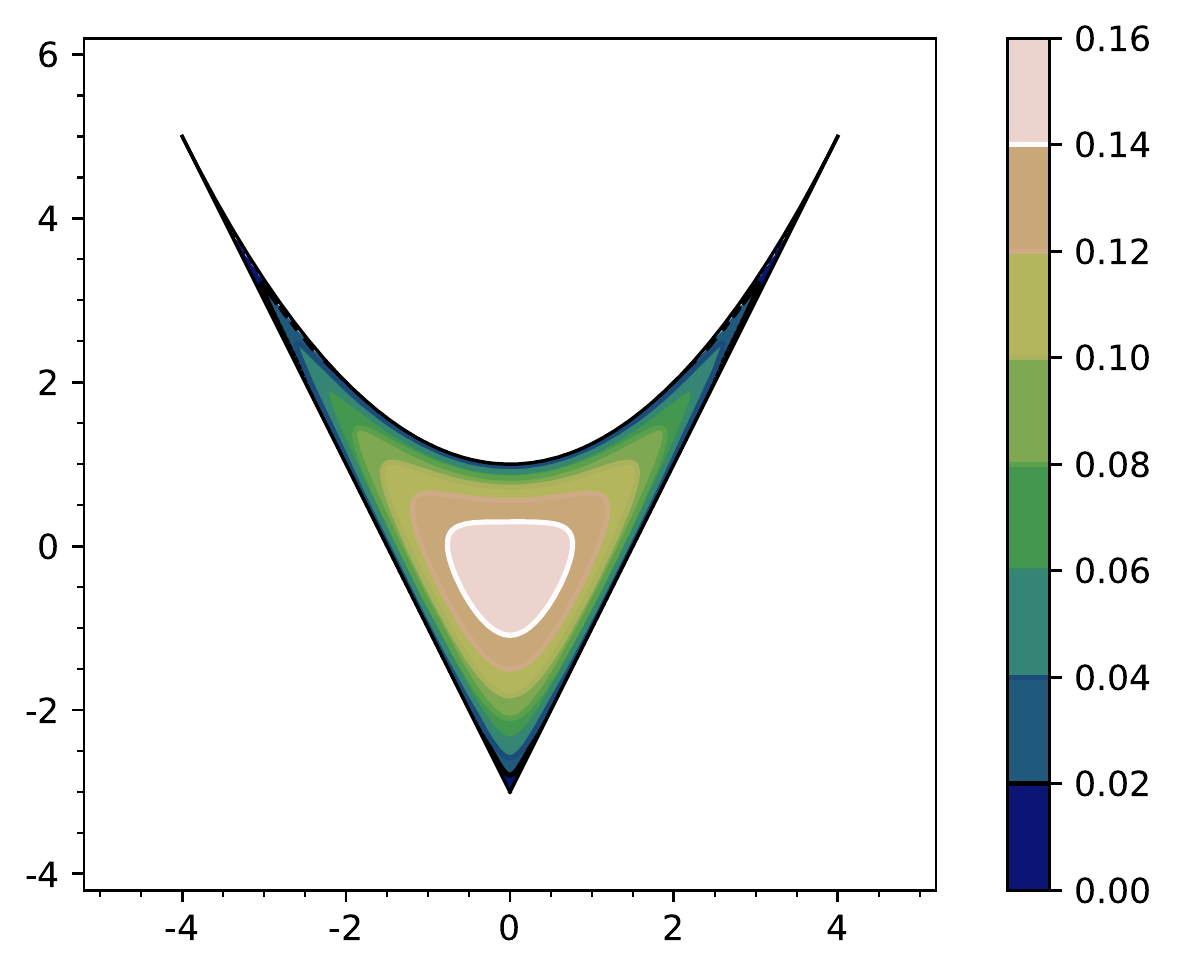} \quad
\includegraphics[height=2.2in]{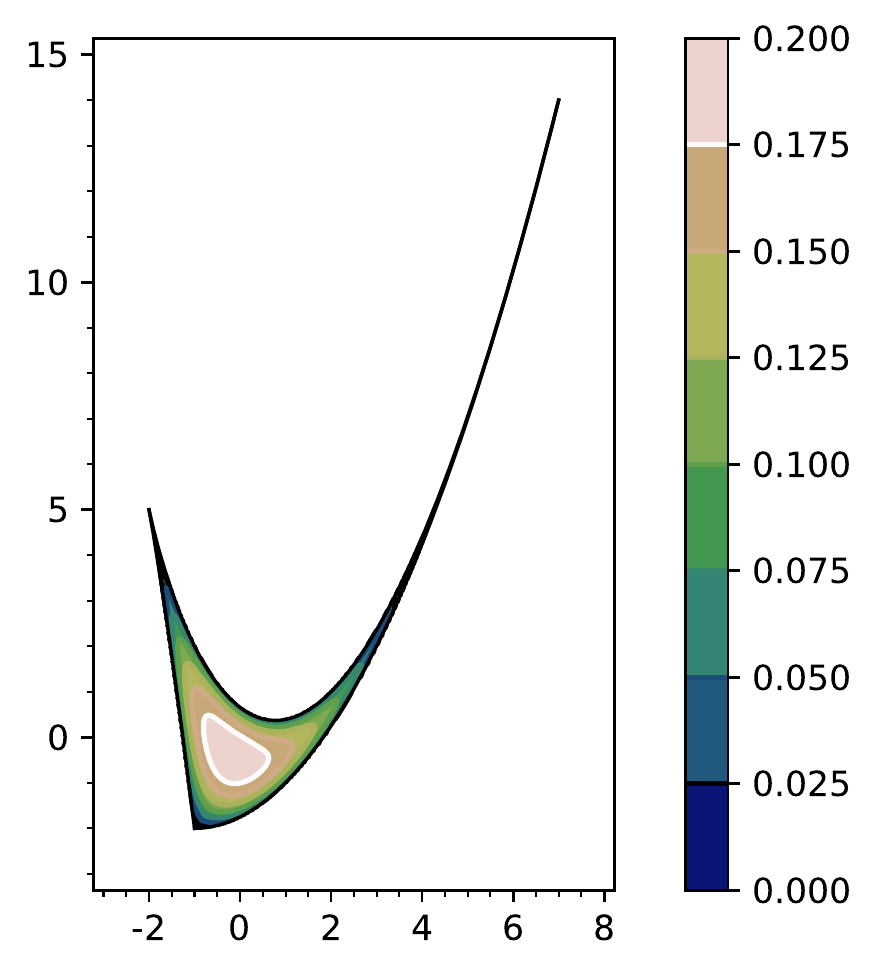}
\caption{Contour plots of the density $\varphi$ of the pushforward of the Haar measure with respect to the standard measure.  The group $\Sp(4)$ is on the  left and $G_2$ is on the right.  The axes in each panel  are the same as in Figures \ref{b2.fig} and  \ref{g2.fig}.} \label{haar.fig}
\end{figure}

The expression inside the absolute value in Equation \eqref{varphi.def} is real and invariant under the Weyl group, so it can be rewritten as a function of $x(t) = \chi_1(t)$ and $y(t) = \chi_2(t)$ --- call it $D(x, y)$.  Doing this for type $\Ct_2$, we find:
\[
D(x,y) = (x^2+4-4y)(-2x-3-y)(2x-3-y),
\]
which is (up to a scalar factor) the product of the equations for the boundary curves.  The maximum of $D$ occurs at $D(0,-1/3)=1024/27$, i.e., $\max(\vf(t)) = 2^3/(\pi^2 3^{3/2})$.  Interestingly, this maximum occurs  at $t = \exp(t_1, 1/2)$, where $t_1 \approx 0.348$  appears to be an algebraic integer that is irrational.  In particular $t$ appears to have infinite order as an element of $T$.

Performing the same computation for $\Gt_2$, we find:
\[
D(x,y) = (x^2 +  2x -  7 - 4y)(y^2 + 10y-7-4x^3 +x^2 + 2x+10xy).
\]
The maximum of this  $D$ occurs at 
\[
D(-1/5, -2/5) = 2^8  3^6 5^{-5} \approx 60,
\]
  i.e.,  
  \[
  \max_{t \in T} \, \vf(t) = 2^23^3/(\pi^2 5^{5/2}) \approx 0.2.
  \]   Note that this function has very small values as we move along the arm that reaches out to $\vf(1_G)$ at $(7, 14)$.  Indeed, the integral of $\varphi$ over the end of the arm, for $x \ge 6$, is less than $10^{-6}$, as opposed to the integral over the whole region, which is 1.

\section{Conclusion}

We have now exhibited two ways to draw the picture $\vf(G)$ of a compact Lie group $G$.  One is approximate in nature, and involves randomly generating points in $\vf(G)$, which we described as painting with a shotgun.  Another is rigorous, in that $\vf(G)$ is the image of a simplex, and we exhibited how to write down equations for the images of each of the faces of the simplex.   We worked through all the details for each of the three rank 2 cases.  We also plotted the push forward of the Haar measure to $\vf(G)$. 

We leave it to the interested reader to take these results further.  For example, one could plot the pictures in $\R^3$ for the three rank 3 simple and simply connected compact groups, $\SU(4)$, $\Spin(7)$, and $\Sp(6)$.  Alternatively, one could look at split semisimple groups instead of compact ones, and therefore consider also fields other than $\R$.
In yet another direction, there are subobjects of $G$ that are unique up to $G$-conjugacy.  For example, there is a unique rank 1 subtorus  that is a maximal torus of a principal $A_1$.  For another, among the elements of finite order in $T$ there is Kostant's principal element from \cite[\S8.6]{Kost:prin} (cf.~\cite{Prasad}), which has the smallest possible order of a regular torsion element in $T$, namely the Coxeter number.  One could plot the image of each of these in $\vf(G)$.

\subsection*{Acknowledgments} Thank you to Serre for interesting conversations on this subject and to Mikhail Borovoi and Drew Sutherland for providing references and pointers.

\providecommand{\bysame}{\leavevmode\hbox to3em{\hrulefill}\thinspace}
\providecommand{\MR}{\relax\ifhmode\unskip\space\fi MR }
\providecommand{\MRhref}[2]{%
  \href{http://www.ams.org/mathscinet-getitem?mr=#1}{#2}
}
\providecommand{\href}[2]{#2}


\end{document}